\documentclass[journal,twoside,web]{ieee}

\usepackage{amsthm}

\usepackage{generic}
\usepackage{cite}
\usepackage{amsmath,amssymb,amsfonts}
\allowdisplaybreaks
\usepackage{algorithmic}
\usepackage{graphicx}
\usepackage{textcomp}
\usepackage{lipsum}
\usepackage{stfloats}
\usepackage{hyperref}
\PassOptionsToPackage{dvipsnames,svgnames,table}{xcolor}
\let\labelindent\relax
\usepackage{enumitem}  
\usepackage{microtype}
\usepackage{booktabs}
\usepackage{bigints}
\usepackage{mathrsfs}
\usepackage{mathtools}

\usepackage{standalone}

\usepackage{color}

\usepackage{multirow}
\usepackage{upgreek}
\usepackage{accents}

\newcommand{\ov}{\xbar}

\newcommand{\T}{\top}

\newtheorem{remark}{Remark}{}

\newtheorem{thm}{Theorem}

\newtheorem{lem}{Lemma}
\newtheorem{defin}{Definition}
\newtheorem{assump}{Assumption}

\newcommand*\xbar[1]{%
   \hbox{%
     \vbox{%
       \hrule height 0.7pt 
       \kern0.35ex
       \hbox{%
         \kern-0.1em
         \ensuremath{#1}%
         \kern-0.1em
       }%
     }%
   }%
} 

\begin{document}
\title{Dissipativity-based $\mathcal{L}_2$ gain-scheduled static
output feedback design for rational LPV systems}
\author{Valessa V. Viana, Diego de S. Madeira, and Thiago Alves Lima
\thanks{Valessa V. Viana, Diego de S. Madeira, and Thiago Alves Lima are with the Departamento de Engenharia Elétrica, Universidade Federal do Ceará, Fortaleza, Brazil. }}

\maketitle

\begin{abstract}
This paper proposes the design of gain-scheduled static output feedback controllers for the stabilization of continuous-time linear parameter-varying systems with $\mathcal{L}_2$-gain performance. The system is transformed into the form of a differential-algebraic representation which allows dealing with the broad class of systems whose matrices can present rational or polynomial dependence on the parameter. The proposed approach uses the definition of strict QSR-dissipativity, Finsler’s Lemma, and the notion of linear annihilators to formulate conditions expressed in the form of polytopic linear matrix inequalities for determining the gain-scheduled static output feedback control for system stabilization. One of the main advantages of the strategy is that it provides a simple design solution in a non-interactive manner. Furthermore, no restriction on the plant output matrix is imposed. Numerical examples highlight the effectiveness of the proposed method.
\end{abstract}


\section{Introduction}
\label{intro}

Static output feedback (SOF) design is a very important problem in control theory. In some cases, a feedback controller that uses all system state information can not be applied due to the impossibility of measuring all the states of the system. Then, an output feedback controller using only the available states has to be designed \cite{sadabadi2016static}. The SOF design is a challenging problem since its mathematical formulation leads to non-convex conditions which can not be solved by semidefinite programming (SDP) \cite{trof2}.  Even though there are many strategies that provide solutions to this problem, it is known that a definitive solution is yet to appear \cite{sadabadi2016static}. See \cite{vese1,trof2,gahinet2011structured,sadabadi2016static} for an overview of the subject.

In the control of linear parameter-varying (LPV) systems, the gain-scheduling technique has received significant attention in the last decades \cite{wei2014survey,apkarian1995self}. The gain-scheduling approach is based on the measurement of the time-varying parameter which adjusts the controller gains for the complete range of parameter variation \cite{shamma1991guaranteed,rugh2000research}. It is well known that a gain-scheduling approach provides less conservative results for LPV systems compared with others control methods \cite{apkarian1995self,shamma1991guaranteed,peres2005}.

In the literature, there are few works that use SOF gain-scheduling techniques for LPV systems. Methods for the gain-scheduled SOF (GS-SOF) design for continuous-time LPV systems have been published in \cite{al2018static,nguyen2018gain}. Recently, a GS-SOF design for LPV systems was developed in a two stage method, where it is necessary designing a non-scheduled static feedback in the first stage \cite{sereni2019new}. In \cite{behrouz2021robust}, a GS-SOF non-iterative design procedure with \(\mathcal{H}_2/\mathcal{H}_{\infty}\) performance has been developed. However, as it is common in the field, these strategies consider the polytopic approach, then the LPV system can only be affine on the parameter. Few works consider a polynomial or rational dependence on the parameter. A paper considering this type of dependence is \cite{bouali2007new}, where a state feedback design was developed for rational LPV systems. In \cite{bouali2008}, gain-scheduled dynamic output feedback design with \(\mathcal{H}_2\) performance has been proposed for rational LPV systems. In \cite{masubuchi2008gain}, a procedure for designing dynamic gain-scheduled controllers for rational LPV systems in the descriptor form was developed. Recently, \cite{polcz2020induced} proposes a novel method to compute the \(\mathcal{L}_2\)-gain for rational LPV systems, however no control law is designed. Many solutions to the gain-scheduled static output feedback design for discrete-time LPV systems have also been recently developed \cite{sadeghzadeh2017gain,de2010gain,peixoto2020,peixoto2021improved}. However, none of them consider rational dependence on the parameter

 Dissipativity theory was introduced some decades ago and has been extensively used in stability analysis and control systems design \cite{willm2,brogl1}. Recently, \cite{madeira2021} proved, under mild assumptions, that a specific case of dissipativity called strict QSR-dissipativy is a necessary and sufficient condition for SOF stabilizabity of LTI systems. In \cite{madeira2020application}, the same concept of dissipativity has been used to the linear SOF design for uncertain nonlinear systems. 

In this paper, we develop a strategy based on some ideas presented in \cite{madeira2021}.
Our strategy provides sufficient conditions for the design of a gain-scheduled SOF that stabilizes continuous-time LPV systems with rational or polynomial dependence. Furthermore, we also consider the influence of external signals on the system and propose the stabilization with \(\mathcal{L}_2\)-gain performance.
The strategy here uses strict QSR-dissipativity, Finsler's Lemma and a differential algebraic representation (DAR) for the LPV system to obtain polytopic LMI conditions for the SOF design. The main contribution of this paper is to consider others types of dependencies on the parameter, such as rational, or polynomial, since most works only consider an affine dependence. It is important to highlight that no restriction on the output plant matrix is imposed. Finally, differently from most papers dealing with SOF design, the proposed strategy is non-iterative and is solved in only one stage. 

This paper is organized as follows. In Section 2, we present important theoretical preliminaries for the formulation of developed conditions. In Section 3, the proposed strategy for gain scheduling SOF stabilization is introduced. Extension of the proposed method to the \(\mathcal{L}_2\)-gain performance case is realized within Section 4. In Section 5, some numerical examples are provided to illustrate the efficiency of the strategy. Finally, in Section 6, we have the conclusion of the paper.

 \noindent \textbf{Notation.} For a matrix \(H \in \mathbb{R}^{n \times m}\), \(H^\top \in \mathbb{R}^{m \times n}\) means its transpose. Operators \(H \succ 0\) and \(H \succeq 0\) mean that the symmetric matrix \(H\) is positive definite or positive semidefinite, respectively. \(He\{A\}\) stands for \(A+A^\top\). \textbf{1}, I, 0 and \(J\) denote all-ones, identity, null, and exchange matrices (i.e., anti-diagonal matrix with ones) of appropriate dimensions, which can be explicitly presented when relevant. For a symmetric block matrix, the symbol \(\star\) stands for the transpose of the blocks outside the main diagonal block. Additionally, for matrices \(A\) and \(H\), \(diag(A, H)\) corresponds to the block-diagonal matrix. \(\nabla\) represents the gradient function. \(\mathbb{R}^+\) denotes the set of elements \(\beta \in \mathbb{R}\) such that \(\beta \geq 0\). Finally, $\| f \|_2$ is used to denote the $l_2$ norm of $f(t): \mathbb{R}^+ \rightarrow \mathbb{R}^n $, given by $\sqrt{
 \int_{0}^{t} f^\top(\tau) f(\tau) d\tau}$. 

\section{Preliminaries}
\subsection{LPV systems}
Consider an LPV system of the form
\begin{align}
    \begin{cases}\label{LPVsystem}
  \dot{x}(t)=A(\rho)x(t)+B(\rho)u(t),\\
  y(t)=C(\rho)x(t),
    \end{cases}
\end{align}
where \(x(t) \in \mathbb{R}^{n}\) is the state vector, \(u(t)\in \mathbb{R}^{m}\) is the control input, \(y(t)\in \mathbb{R}^{p}\) is the measured output. Moreover, \(\rho \in \Omega \subset \mathbb{R}^r\) is a vector of time-varying parameters and \(A(\rho)\in \mathbb{R}^{n \times  n}, B(\rho)\in \mathbb{R}^{n \times  m}, C(\rho)\in \mathbb{R}^{p \times  n}\) are polynomial or rational matrices on \(\rho\).
\begin{assump}
The elements of the parameters vector are bounded and the vector \(\rho\) lies inside a polytope \(\Omega\) of \(N=2^r\) vertices, where \(r\) is the number of elements of \(\rho\). The polytope \(\Omega\) is given by
\begin{eqnarray} \label{eq:polytope} 
\hspace{-0.2cm}&\Omega=\{\alpha(\rho(t)) \in \mathbb{R}^{N}:\sum \limits_{i=1}^{N} \alpha_i=1; \alpha_i \geq 0; i=1,\ldots,N\},
\end{eqnarray}
where any point inside \(\Omega\) can be represented by the convex combination of its vertices \cite{briat2014}.

\end{assump}

\subsection{Differential-Algebraic Representation - DAR}
The LPV system \eqref{LPVsystem} can be described by a Differential Algebraic Representation (DAR), as presented in \cite{coutinho2009robust}. A Differential Algebraic Representation is given by 
\begin{align}
    \begin{cases}\label{DAR}
\dot{x}=A_1x + A_2 \pi + A_3u,\\
y=C_1x+C_2\pi,\\
0=\Upsilon_1(\rho)x+\Upsilon_2(\rho)\pi+\Upsilon_3(\rho)u,
    \end{cases}
\end{align}
where \(\pi(x,\rho,u) \in \mathbb{R}^{n_\pi}\) is an auxiliary vector that contains all nonlinear terms of \eqref{LPVsystem} depending on \(\rho\). \(A_1 \in \mathbb{R}^{n \times n},A_2\in \mathbb{R}^{n \times n_{\pi}},A_3\in \mathbb{R}^{n \times m},C_1\in \mathbb{R}^{p \times n},C_2\in \mathbb{R}^{p \times n_{\pi}}\) are constant matrices and \(\Upsilon_1(\rho) \in \mathbb{R}^{n_\pi \times n},~\Upsilon_2(\rho) \in \mathbb{R}^{n_\pi \times n_\pi},~\Upsilon_3(\rho) \in \mathbb{R}^{n_\pi \times m}\) are affine matrices of \(\rho\).

The DAR of a system is not unique and a state-space
representation \eqref{LPVsystem} is well-posed in its DAR form if \(\Upsilon_2(\rho)\) is invertible since from \eqref{DAR} we have
\begin{eqnarray}\label{eq:systeminvertible}
&\pi(x,u,\rho)=\Upsilon_2^{-1}[\Upsilon_1x-\Upsilon_3u],\\
&\dot{x}=(A_1-A_2\Upsilon_2^{-1}\Upsilon_1)x+(A_3-A_2\Upsilon_2^{-1}\Upsilon_3)u.
\end{eqnarray}
\begin{remark}
The DAR \eqref{DAR} is an alternative and exact representation of system \eqref{LPVsystem}. It is important to highlight that it can model the whole class of LPV systems with rational dependence on the parameters without singularities at the origin \cite{coutinho2009robust}.
A general procedure to obtain the DAR of the LPV system can be found in \cite{coutinho2009robust}.
\end{remark}

The motivation to represent the LPV system in a DAR form is that, in \eqref{DAR} the system matrices \(A_i\) and \(C_i\) are constant and the dependency on \(\rho\) is transferred to the auxiliary matrices \(\Upsilon_i(\rho)\). Moreover, the auxiliary matrices depend only affinely on \(\rho\), which allows the use of techniques leading to convex design conditions expressed in the form of LMIs in this work.

\subsection{Finsler's Lemma}

A version of Finsler's Lemma from \cite{trofi2} is presented in the following Lemma.
\begin{lem} \label{finslerlemma}
Consider \(\mathcal{W} \subseteq \mathbb{R}^{n_s}\) a given polytopic set, and let \(Q_d:\mathcal{W} \rightarrow \mathbb{R}^{n_q \times n_q}\) and \(C_d:\mathcal{W} \rightarrow \mathbb{R}^{n_r \times n_q}\) be given matrix functions, with \(Q_d\) symmetric. Then, the following statements are equivalent 
\renewcommand{\theenumi}{\roman{enumi}}%
\begin{enumerate}
    \item \(\forall w \in \mathcal{W}\) the condition that \(z^{\top}Q_d(w)z>0\) is satisfied \(\forall z \in \mathbb{R}^{n_q}:C_d(w)z=0\). 
    \item \(\forall w \in \mathcal{W}\) there exists a certain matrix function \(L:\mathcal{W} \rightarrow \mathbb{R}^{n_q \times n_r}\) such that \(Q_d(w)+L(w)C_d(w)+C_d(w)^{\top}L(w)^{\top}\succ 0\).
\end{enumerate}
\end{lem}
If \(C_d\) and \(Q_d\) are affine functions of \(w\),  and \(L\) is a constant matrix to be determined, then \(ii)\) becomes a polytopic LMI condition which is sufficient for \(i)\). Lemma \ref{finslerlemma} also applies for testing negative definite functions. Clearly, \(C_d\) is an annihilator of the vector \(z\), which is not unique. Further details and a systematic procedure for determining linear annihilators are presented in \cite{trofi2} and \cite{coutinho2008}. 

\subsection{Dissipativity}
Consider an LTI system such as
\begin{align}
    \begin{cases}\label{nonlinearsystem}
\dot{x}(t)=Ax(t)+Bu(t),\\
y(t)=Cx(t).
    \end{cases}
\end{align}
System \eqref{nonlinearsystem} is said to be dissipative if it is completely reachable and there exists a nonnegative storage function \(V(x)\), where \(V: \mathbb{R}^n \rightarrow \mathbb{R}\) and \(V \in C^1\), and a locally integrable supply rate \(r(u(t),y(t))\) such that \(\dot{V} \leq r(u,y)\) \cite{haddad2011nonlinear}. Some definitions of dissipativity can be found in \cite{brogl1}. In this work, we use the definition of strict QSR-dissipativity given below.
\begin{defin}\label{defdissip}
A system is said to be strictly QSR-dissipative
along all possible trajectories of \eqref{nonlinearsystem} starting at \(x(0)\), for all \(t \geq 0\), if there exists \(T(x) > 0\) such that
\begin{equation}\label{dissipativityeq_geral}
  \dot{V}(x)+T(x) \leq y^{\top}Qy+2y^{\top}Su+u^{\top}Ru,
 \end{equation}
 where \(S \in \mathbb{R}^{p \times m}\) is real and \(Q \in \mathbb{R}^{p \times p},R \in \mathbb{R}^{m \times m}\) are real and symmetric.
 \end{defin}

 From a practical point of view, a dissipative system stores only a fraction of the energy supplied to it through \(r(u,y)\) and only a fraction of its stored energy \(V(x)\) can be delivered to its surroundings. Definition \ref{defdissip} can be related with Lyapunov stability. If a system is strictly QSR-dissipative with \(V(x) > 0\) and \(Q \preceq 0\), then the free system is asymptotically stable \cite{haddad2011nonlinear}.
 
 In this work, we consider quadratic Lyapunov functions 
\begin{equation}\label{Def_V}
    V(x)=x^{\top}Px,~~P \succ 0,
\end{equation}
 where \(P \in \mathbb{R}^{n \times n}\), and a quadratic \(\rho\)-parameter dependent function \(T(x,\rho)\) that can be defined in a polytopic domain
 \begin{equation}\label{Def_T}
T(x,\rho)=x^{\top}H(\rho)x,~~H(\rho)=\sum \limits_{i=1}^{N}  \alpha_i H_i,~H_i \succ 0,
 \end{equation}
where \(H_i \in \mathbb{R}^{n \times n}\). Also, we consider \(Q\) and \(S\) \(\rho\)-parameter dependent matrices in a polytopic domain,
\begin{eqnarray}\label{Def_QeS}
&Q(\rho)= \sum \limits_{i=1}^{N}  \alpha_i Q_i,~~S(\rho)= \sum \limits_{i=1}^{N}  \alpha_i S_i.
\end{eqnarray}
where \(Q_i \in \mathbb{R}^{p \times p}\) and \(S_i \in \mathbb{R}^{p \times m}\). Thus, considering \eqref{Def_V}-\eqref{Def_T}-\eqref{Def_QeS}, a version of the dissipativity condition \eqref{dissipativityeq_geral} for the case of LPV systems \eqref{LPVsystem} in a DAR form such as \eqref{DAR} is given by 
\begin{equation}
    \begin{gathered}\label{dissipativityeq_lpv}
  t_d(x,u,\rho)=\nabla V^{\top}[A_1x+A_2\pi+A_3u]\\
  +x^{\top}H(\rho)x-y^{\top}Q(\rho)y-2y^{\top}S(\rho)u-u^{\top}Ru \leq 0.
    \end{gathered}
\end{equation}
The system \eqref{LPVsystem} is said to be \textit{robust strictly QSR-dissipative} if \eqref{dissipativityeq_lpv} holds for all \(\rho \in \Omega\).

\section{GS-SOF Stabilization}

The strategy proposed in this work consists in connecting Lemma \ref{finslerlemma} and dissipativity condition \eqref{dissipativityeq_lpv}, assuming parameter dependent matrices on the supply rate and on function \(T\).  In order to apply Lemma \ref{finslerlemma}, we consider the following notation
 \begin{eqnarray}\label{s3:eq1} \nonumber
 &w=\rho(t),~\mathcal{W}=\Omega,~n_s=r,\\ \nonumber
 &n_r=n_\pi,~n_q=n+n_\pi+m.
 \end{eqnarray}
 Next, observe that \(t_d(x,u,\rho)\) from \eqref{dissipativityeq_lpv} can be decomposed in the following manner
\begin{eqnarray}\label{s3:eq4}
&t_d(x,u,\rho)=\pi_d^{\top}Y(\rho)\pi_d,\\\label{s3:eq5} \nonumber
&\pi_d=\begin{bmatrix}
   x^{\T}
   & \pi^{\T} &
   u^{\T}
\end{bmatrix}^{\T},~~Y(\rho)=\sum \limits_{i=1}^{N}  \alpha_i Y_{i},
\end{eqnarray}
where \(Y_{i}\) is a symmetric and linear matrix on all the unkown coefficients of \((Q_i,S_i,R,P)\). In addition, consider
\begin{eqnarray}\label{s3:eq6}
C_d(\rho)=\begin{bmatrix}
   \Upsilon_1(\rho) & \Upsilon_2(\rho) & \Upsilon_3(\rho)
\end{bmatrix}
\end{eqnarray}
as a linear annihilator of \(\pi_d\). Since matrices \(\Upsilon(\rho)\) are affine on the parameter, these matrices can be represented in a polytopic domain, leading to the following representation of \(C_d(\rho)\) 
\begin{eqnarray}\label{s3:eq7}
C_d(\rho)=\sum \limits_{i=1}^{N}  \alpha_i C_{d_i}=\sum \limits_{i=1}^{N}  \alpha_i\begin{bmatrix}
   \Upsilon_{1_i} & \Upsilon_{2_i} & \Upsilon_{3_i}
\end{bmatrix}.
\end{eqnarray}

The following theorem provides a solution for the design of a gain-scheduled SOF that stabilizes LPV systems.

\begin{thm}\label{main:thm}
 Let \(\Omega\) be a polytope of \(\rho(t)\) described by \eqref{eq:polytope} and \(C_d(\rho)\) a linear annihilator of \(\pi_d\) described by \eqref{s3:eq7}. Given a scalar \(\beta\), assume there exist symmetric matrices \(P \succ 0 ,~H_i \succ 0 ,~R \succ 0 ,~Q_{i} , \text{ and matrices } S_{i} \), \(L \in \mathbb{R}^{n_q \times n_\pi}\), such that 
\begin{equation}\label{teo:eq1}
    Y_{i} + LC_{d_i} + C_{d_i}^{\top}L^{\top} \prec 0, 
\end{equation}
\begin{equation} \label{19cond}
\begin{split}
  X_{d_i} + L_s C_{s_i} + C_{s_i}^\top L_s^\top \prec 0
\end{split}
\end{equation} 
\noindent for \(i=1,\ldots,N\), where \(L_s=[\beta \textbf{1} ~ -\text{I}]^\top,~C_{s_i}=[S_i^\top ~ R]\),
\begin{equation*}
    X_{d_i}=\begin{bmatrix}
       Q_i & S_i\\
       S_i^\top & R
    \end{bmatrix},
\end{equation*}
and
\begin{equation}\label{teo:eq2} \nonumber
Y_i=  \begin{bmatrix}
    PA_1+A_1^{\T}P-C_1^{\top}Q_{i}C_1 + H_i &   \star & \star\\
    \left(PA_2-C_1^{\top}Q_{i}C_2\right)^{\T}  & -C_2^{\top}Q_{i}C_2 & \star\\
    \left(PA_3-C_1^{\top}S_{i}\right)^{\T}  & -S_{i}^{\T}C_2 & -R
    \end{bmatrix}.
\end{equation}
Then system \eqref{LPVsystem} is \textit{robust strictly QSR-dissipative} for all \(\rho(t) \in \Omega\) and the gain-scheduled SOF
\begin{eqnarray}\label{teo:eq4}
u=K(\rho)y,~K(\rho)=\sum \limits_{i=1}^{N}  \alpha_i K_i,~K_i=-R^{-1}S_i^{\top},
\end{eqnarray}
\noindent asymptotically stabilizes \eqref{LPVsystem}, for all \(\rho \in \Omega\), around the origin.
\end{thm}

\begin{proof}
First, consider the satisfaction of condition \eqref{teo:eq1}. Since $\alpha_i\geq0$ and $\sum \limits_{i=1}^{N}  \alpha_i=1$ for $i=1,\dots,N$, note that, by multiplying all the terms of \eqref{teo:eq1} by \( \alpha_i\) and summing them up from \(i=1\) to \(i=N\), we obtain
\begin{equation}\label{eq:proofrho}
\begin{split}
    &\sum \limits_{i=1}^{N}  \alpha_i( Y_i + \text{He}\{LC_{d_i}\} ) = Y(\rho) + \text{He}\{LC_d(\rho)\} \prec 0.
\end{split}
\end{equation}
Since \(C_d(\rho)\) is an annihilator of \(\pi_d\), from Lemma \ref{finslerlemma}, satisfaction of \eqref{eq:proofrho} implies that \(\pi_d^{\T} Y(\rho) \pi_d = t_d(x,u,\rho) < 0\) is also satisfied for all \(\rho \in \Omega\) and all \(\pi_d \neq 0\), where $t_d(x,u,\rho)$ was first defined in \eqref{dissipativityeq_lpv}. Thus, system \eqref{LPVsystem} is \textit{robust strictly QSR-dissipative} for all \(\rho \in \Omega\). In addition, note that as \(H(\rho) \succ 0\), fulfulling 
\begin{equation}\label{eq:2}
    y^{\top}Q(\rho)y+2y^{\top}S(\rho)u+u^{\top}Ru \leq 0
\end{equation}
is sufficient to guarantee $\nabla V^{\top}[A_1x+A_2\pi+A_3u] < 0 $, which ensures the asymptotic stability of system \eqref{LPVsystem} about the origin. Considering a vector \(\zeta=[y^\top ~ u^\top]^\top\), condition \eqref{eq:2} can be rewritten as \(\zeta^\top X_d(\rho) \zeta \leq 0\), where 
\begin{equation} \nonumber
    X_d(\rho)= \begin{bmatrix}
       Q(\rho) & S(\rho)\\
       S^\top(\rho) & R
    \end{bmatrix}.
\end{equation}
Let us recall that, we consider the gain-scheduled static output feedback given by
 \begin{equation}
  \begin{gathered}\label{controllaw}
 u=K(\rho)y=-R^{-1}\sum \limits_{i=1}^{N}  \alpha_i S_i^{\top}y=-R^{-1}S^{\top}(\rho)y.
 \end{gathered}
\end{equation}
By noting that \(C_s(\rho) \zeta=0\), with \(C_s(\rho)=[S^\top(\rho)~R]\), Lemma \ref{finslerlemma} can be applied. If there exists matrix \(L_s\) such that
\begin{equation}\label{condstability}
\begin{split}
    X_d(\rho) + \text{He}\{L_s C_s(\rho)\} \prec 0,
\end{split}
\end{equation}
then \(\zeta^\top X_d(\rho) \zeta < 0\) for all \(\rho \in \Omega\) and \(\zeta \neq 0\), thus condition \eqref{eq:2} is also satisfied, ensuring $\nabla V^{\top}[A_1x+A_2\pi+A_3u] < 0 $. Note that by multiplying all the terms of \eqref{19cond} by \( \alpha_i\) and summing them up from \(i=1\) to \(i=N\), we obtain
\begin{equation}
    \sum \limits_{i=1}^{N}  \alpha_i(X_{d_i}+ \text{He}\{L_s C_{s_i}\} ) = X_d(\rho) + \text{He}\{L_s C_s(\rho)\} \prec 0.
\end{equation}
Therefore, satisfaction of \eqref{19cond} implies fulfillment of \eqref{condstability} and that the system is stabilized by the SOF gain-scheduling given by \eqref{teo:eq4}, which completes the proof of Theorem \ref{main:thm}.
\end{proof}

\section{$\mathcal{L}_2$-gain Performance}
In this section, we present an extension of the proposed GS-SOF design procedure to the case of \(\mathcal{L}_2\)-gain performance when the system is affected by external disturbances.  The approach presented here is inspired by the framework proposed in \cite{cao1998static} for stabilization of an LTI system with \(\mathcal{L}_2\)-gain performance.

First, consider an LPV system of the form
\begin{align} \label{sys_disturbance}
    \begin{cases}
\dot{x}(t)=A(\rho)x(t)+B (\rho)u(t)+B_w (\rho)w(t),\\
z(t)=A_z(\rho)x(t)+B_z (\rho)u(t)+D_z (\rho)w(t),\\
y(t)=C(\rho)x(t)+D(\rho)w(t),    
    \end{cases}
\end{align}
that is the same system \eqref{LPVsystem} with an additional external input \(w(t) \in \mathbb{R}^q\) and a controlled output \(z(t) \in \mathbb{R}^l\). As in \eqref{LPVsystem}, all system matrices can present rational or polynomial dependence on \(\rho\). This system in its DAR form is presented below
\begin{align}\label{dar_w}
    \begin{cases}
\dot{x}=A_1x + A_2 \pi + A_3u + A_4 w,\\
z=B_1x + B_2 \pi + B_3u + B_4 w,\\
y=C_1x+C_2\pi+C_3w,\\
0=\Upsilon_1(\rho)x+\Upsilon_2(\rho)\pi+\Upsilon_3(\rho)u+\Upsilon_4(\rho)w,    
    \end{cases}
\end{align}
where matrices \(B_i\) are also constant matrices. Considering \(u=K(\rho)y\), the closed loop form of this system is given by
\begin{equation} \label{closedloop}
    \begin{cases}
\dot{x}=\mathscr{A}_1x + \mathscr{A}_{2} \pi + \mathscr{A}_{3} w,\\
z=\mathscr{B}_{1}x + \mathscr{B}_{2} \pi + \mathscr{B}_{3}w,\\
0=\hat{\Upsilon}_{1}(\rho)x+\hat{\Upsilon}_{2}(\rho)\pi+\hat{\Upsilon}_{3}(\rho)w,    
    \end{cases}
\end{equation}
where
\begin{equation}\label{rel_closed}
    \begin{split}
&\mathscr{A}_{1}=(A_1 + A_3K(\rho)C_1),~\mathscr{A}_{2}=(A_2 +A_3K(\rho)C_2),\\
&\mathscr{A}_{3}=(A_4 +A_3K(\rho)C_3),~\mathscr{B}_{1}=(B_1 + B_3K(\rho)C_1),\\
&\mathscr{B}_{2}=(B_2 + B_3K(\rho)C_2),~\mathscr{B}_{3}=(B_4 + B_3K(\rho)C_3),\\
&\hat{\Upsilon}_{1}=(\Upsilon_1+ \Upsilon_3K(\rho)C_1),~\hat{\Upsilon}_{2}=(\Upsilon_2+ \Upsilon_3K(\rho)C_2),\\
&\hat{\Upsilon}_{3}=(\Upsilon_4+ \Upsilon_3K(\rho)C_3).
    \end{split}
\end{equation}

The gain-scheduled static output feedback control problem with \(\mathcal{L}_2\)-gain performance is equivalent to finding a control law \(u(t)=K(\rho(t))y(t)\) such that the closed loop \eqref{closedloop} is asymptotically stable in the absence of disturbance \(w\) and the \(\mathcal{L}_2\) norm of \(z\) is bounded such that
\begin{equation} \label{L2_rel}
   \|z\|_2 \leq \gamma \|w\|_2 + \theta,
\end{equation}
with positive scalars \(\gamma\) and \(\theta\), where \(\theta\) is a bias term. When \eqref{L2_rel} is ensured, one can say that the system \eqref{closedloop} is input to output stable with \(\mathcal{L}_2\)-gain bounded by \(\gamma\). In order to guarantee asymptotic stability at the same time satisfying relation \eqref{L2_rel}, we have the following sufficient condition \cite{coutinho2008}
\begin{equation}\label{L2gain}
    \dot{V}+\gamma^{-1}z^\top z - \gamma w^\top w < 0,
\end{equation}
with function \(V\) defined in \eqref{Def_V}. Note that by integrating both sides of \eqref{L2gain}, taking squares roots, and using the fact that \(\sqrt{a+b} \leq a+b\), for \(a,b \in \mathbb{R}^+\), one arrives at \( \|z\|_2 \leq \gamma \|w\|_2 + \sqrt{\gamma V(x(0))}\), i.e., \eqref{L2_rel} with bias term \(\theta=\sqrt{\gamma V(x(0))}\).
\begin{thm}
 If there exists a scalar \(\gamma>0\), such that conditions \eqref{teo:eq1} and \eqref{19cond} of Theorem \ref{main:thm} hold replacing matrices \((P,A_1,A_2,A_3,C_1,C_2,\Upsilon_1,\Upsilon_2,\Upsilon_3,L,H_i)\) by \((\ov{P},\ov{A}_1,\ov{A}_2,\ov{A}_3,\ov{C}_1,\ov{C}_2,\ov{\Upsilon}_1,\ov{\Upsilon}_2,\ov{\Upsilon}_3,\ov{L},\ov{H}_i)\), respectively,  where 

\begin{equation} \label{barmatrices}
\begin{split}
 &\ov{P}=\begin{bmatrix}
       P & 0 & 0\\
       0 & \text{I}_q & 0\\
       0 & 0 & \text{I}_l
    \end{bmatrix},~\ov{A}_1=\begin{bmatrix}
   A_1 & A_4 & 0_{n \times l}\\
   0_{l \times n} & -\frac{\gamma}{2}\text{I}_{l \times q} & 0_{l \times l}\\
    B_1 & B_4 & -\frac{\gamma}{2}\text{I}_{l}
\end{bmatrix},\\
&\ov{A}_2=\begin{bmatrix}
   A_2 \\
   0_{q \times n_\pi}\\
   B_2
\end{bmatrix},~\ov{A}_3= \begin{bmatrix}
   A_3 \\
   0_{q \times m}\\
   B_3
\end{bmatrix},~\ov{\Upsilon}_1^\top=\begin{bmatrix}
   \Upsilon_1 \\ \Upsilon_4 \\ 0_{n_\pi \times l}
\end{bmatrix}\\
&\ov{C}_2=C_2,~\ov{C}_1=\begin{bmatrix}
   C_1 & C_3 & 0_{r \times l}
\end{bmatrix},,~\ov{\Upsilon}_2=\Upsilon_2,\\
&\ov{\Upsilon}_3=\Upsilon_3,~\ov{H}_i \in \mathbb{R}^{n_l \times n_l},~\ov{L} \in \mathbb{R}^{(n_l+n_\pi+m) \times n_\pi},
\end{split}
\end{equation}
 with \(n_l=n+q+l\), then system \eqref{sys_disturbance} is \textit{robust strictly QSR-dissipative} for all \(\rho(t) \in \Omega\), and the gain-scheduled SOF
 \begin{equation}\label{blablala}
     u=K(\rho)y,~K(\rho)=\sum \limits_{i=1}^{N}  \alpha_i K_i,~K_i=R^{-1}S_i^\top,
 \end{equation}
 asymptotically stabilizes system \eqref{sys_disturbance} for all \(\rho(t) \in \Omega\) with \(\mathcal{L}_2\)-gain bounded by \(\gamma\).
\end{thm}
\begin{proof}
First, consider system \eqref{LPVsystem} in its DAR form \eqref{DAR}. Considering \(u=K(\rho)y\), we have the Lyapunov condition \(\dot{V}(x) < 0\) that guarantees asymptotic stability for the closed-loop system \eqref{DAR}, which is  equivalently expressed as
\begin{equation} \label{proof2:eq1}
\begin{bmatrix}
       x\\
       \pi
\end{bmatrix}^\top
\begin{bmatrix}
       He\{PA_1 + PA_3K(\rho)C_1\}  &  & \star\\
       A_2^\top P + C_2^\top K^\top(\rho) A_3^\top P  &  & 0
\end{bmatrix} \begin{bmatrix}
       x\\
       \pi
    \end{bmatrix} < 0.
\end{equation} 
Since \(u=K(\rho)y\) and also because of \eqref{DAR}, matrix \([\Upsilon_1 + \Upsilon_3K(\rho)C_1~~~\Upsilon_2 + \Upsilon_3K(\rho)C_2]\) is an annihilator of \([x^\T~~\pi^\T]^\T\). Thus Lemma \ref{finslerlemma} can be applied. If there exists matrix \(L_a\) such that
\begin{equation}  \label{proof2:eq2}
\begin{split}
&He\left\{\begin{bmatrix}
   PA_1  & PA_2  \\
   0  &  0
\end{bmatrix}
\right\}+He\left\{\begin{bmatrix}
   PA_3KC_1  & PA_3KC_2\\
   0  &  0
\end{bmatrix} 
\right\}\\
&+He\left\{L_a
\begin{bmatrix}
  \Upsilon_1 &  \Upsilon_2  
\end{bmatrix}
\right\}+He\left\{L_a\Upsilon_3 K
\begin{bmatrix}
   C_1 &  C_2
\end{bmatrix}
\right\} \prec 0,
\end{split}
\end{equation}
then \eqref{proof2:eq1} is satisfied for all \(\rho(t) \in \Omega\) and \([x^\top~ \pi^\top]^\top \neq 0\).

On the other hand, asymptotic stability of system \eqref{closedloop} with \(\mathcal{L}_2\)-gain performance is guaranteed fulfilling \eqref{L2gain}, which is equivalent to \( \pi_w^\top Y_w \pi_w < 0\), where \(\pi_w=[x^\top~w^\top~\pi^\top]^\top\) and \(Y_w\) is given by
\begin{equation*} \label{proof2:eq5}
\left[
\setlength\arraycolsep{2pt}
\begin{array}{ccc}
    He\{P\mathscr{A}_1\} + \gamma^{-1} \mathscr{B}_1^\top \mathscr{B}_1 & \star & \star \\
     \mathscr{A}_3^\top P + \gamma^{-1} \mathscr{B}_3^\top \mathscr{B}_3 & \gamma^{-1} \mathscr{B}_3^\top \mathscr{B}_3 - \gamma \text{I}  & \star\\
    \mathscr{A}_2^\top P + \gamma^{-1} \mathscr{B}_2^\top \mathscr{B}_2  & \gamma^{-1} \mathscr{B}_2^\top  \mathscr{B}_3  & \gamma^{-1} \mathscr{B}_2^\top \mathscr{B}_2 
\end{array}
\right].
\end{equation*}
By noting from \eqref{closedloop} that \( \hat{\Upsilon}_w=[  \hat{\Upsilon}_1 ~ \hat{\Upsilon}_3 ~\hat{\Upsilon}_2]\) is an annihilator of \(\pi_w\), Lemma \ref{finslerlemma} can also be applied. If there exists a matrix \(L_w=[L_1^\top~L_3^\top~L_2^\top]^\top\) such that
\begin{equation} \label{proof2:eq6}
Y_w + L_w \hat{\Upsilon}_w + \hat{\Upsilon}_w^\top L_w^\top \prec  0,
\end{equation}
then \( \pi_w^\top Y_w \pi_w < 0\) is satisfied for all \(\rho(t) \in \Omega\) and \(\pi_w \neq 0\).
Next, applying Schur complement in \eqref{proof2:eq6}
followed by a congruence transformation with \(diag(I_2,J_2)\), we obtain 
\begin{equation} \label{proof2:eq7}
\begin{split}
&\begin{bmatrix}
He\{P\mathscr{A}_1\}  & * & * & *\\
\mathscr{A}_3^\top P & - \gamma \text{I}  & * & *\\
\mathscr{B}_1 & \mathscr{B}_3  & -\gamma \text{I} & *\\
\mathscr{A}_2^\top P  & 0 & \mathscr{B}_2^\top & 0
\end{bmatrix}+He\left\{L_b\ov{\Upsilon}_w\right\}  \prec  0,
\end{split}
\end{equation}
where \(L_b=[L_1^\top ~  L_3^\top ~ 0 ~  L_2^\top]^\top\) and \(\ov{\Upsilon}_w=\begin{bmatrix}
 \hat{\Upsilon}_1 &  \hat{\Upsilon}_3 & 0 &  \hat{\Upsilon}_2
\end{bmatrix}\). By taking into account the definitions of matrices \(\mathscr{A}_1,\mathscr{A}_2,\mathscr{A}_3,\mathscr{B}_1,\mathscr{B}_2,\mathscr{B}_3,\hat{\Upsilon}_1,\hat{\Upsilon}_2,\hat{\Upsilon}_3\) in \eqref{rel_closed}, the following equivalent expression for \eqref{proof2:eq7} is obtained in terms of the matrices \(\ov{P},\ov{A}_1,\ov{A}_2,\ov{A}_3,\ov{C}_1,\ov{C}_2,\ov{\Upsilon}_1,\ov{\Upsilon}_2,\ov{\Upsilon}_3\) given in \eqref{barmatrices}\footnote{Dependence on \(\rho\) was omitted for simplicity of notation.} 
\begin{equation}  \label{proof2:eq8}
\begin{split}
&He\left\{\begin{bmatrix}
  \ov{P}~\ov{A}_1  & \ov{P}~\ov{A}_2  \\
  0     &  0
\end{bmatrix}
\right\}+He\left\{\begin{bmatrix}
  \ov{P}~\ov{A}_3K\ov{C}_1  &   \ov{P}~\ov{A}_3K\ov{C}_2\\
  0  &  0
\end{bmatrix} 
\right\}\\
&+He\left\{L_b~
\begin{bmatrix}
  \ov{\Upsilon}_1 &  \ov{\Upsilon}_2  
\end{bmatrix}
\right\}+He\left\{L_b~\ov{\Upsilon}_3 K
\begin{bmatrix}
  \ov{C}_1 &  \ov{C}_2
\end{bmatrix}
\right\}  \prec 0.
\end{split}
\end{equation}
Note that condition \eqref{proof2:eq8} has the same form of condition \eqref{proof2:eq2}. Thus, by applying Theorem \ref{main:thm} with the bar matrices, one ensures satisfaction of  \( \pi_w^\top Y_w \pi_w < 0\), \( \forall \pi_w \in \mathbb{R}^{n+q+n_{\pi}}:\hat{\Upsilon}_w \pi_w=0\), and for all \(\rho(t) \in \Omega\) with the designed SOF gain-scheduled control \eqref{blablala}, which in turn guarantees \eqref{L2gain} along the trajectories of the closed-loop perturbed system \eqref{closedloop} with \(\mathcal{L}_2\)-gain bounded by \(\gamma\). \end{proof}

\subsection{Optimization Problem}
To design the GS-SOF that stabilizes system \eqref{sys_disturbance} while minimizing the \(\mathcal{L}_2\)-gain, the following optimization problem applies.
\begin{equation} \label{optm}
    \begin{split}
        &\text{minimize } \gamma \\
        &\text{subject to } \gamma > 0,~P \succ 0,\\
        &\text{and } \ov{H}_i \succ 0,~\eqref{teo:eq1},~\eqref{19cond} \text{ for } i=1,\ldots,N,
    \end{split}
\end{equation}
where conditions \eqref{teo:eq1} and \eqref{19cond} are applied as in Theorem 2.
\section{Numerical Examples}
This section presents numerical examples to illustrate the effectiveness of the proposed design method. For the
implementation of the design conditions in the theorems, we use conventional SDP tools provided by \cite{lofberg2004yalmip} and \cite{sturm1999using}.
\subsection{Example 1}
Consider the open-loop unstable system from \cite{bouali2008} with an additional external input \(w\) and an adapted output \(y\). 
\begin{equation}\label{ex1:eq1}
    \begin{cases}
&\dot{x}=
\begin{bmatrix}
  \dfrac{\rho^2 +\rho}{\rho+2} & \dfrac{3\rho +4}{\rho+2}\\
  1 & -1
\end{bmatrix}x+
\begin{bmatrix}
  2 \\  1
\end{bmatrix}u + \begin{bmatrix}
  1 \\  0
\end{bmatrix}w ,\\
&z=\begin{bmatrix} 
  1 & 0
\end{bmatrix}x + u,\\
&y= 
\begin{bmatrix} 
  1 & 0
\end{bmatrix}x + w,  
    \end{cases}
\end{equation}
where \(\rho(t) \in [-1.5~~~ 1.5]\). A DAR of system \eqref{ex1:eq1} is given by
\begin{eqnarray}\nonumber
 &\pi=
\begin{bmatrix}
  \dfrac{x_1}{\rho+2}~~\dfrac{\rho x_1}{\rho+2}~~\dfrac{\rho^2 x_1}{\rho+2}~~\dfrac{x_2}{\rho+2}~~\dfrac{\rho x_2}{\rho+2}~~\dfrac{\rho^2 x_2}{\rho+2}
\end{bmatrix}^{\top}, \\ \nonumber
&A_1=
\begin{bmatrix}
    0 & 0\\
  1 & -1 \\
\end{bmatrix},A_2=
\begin{bmatrix}
     0 & 1 & 1 & 4 & 3 & 0 \\
 0 & 0 & 0 & 0 & 0 & 0 
\end{bmatrix},A_3=
\begin{bmatrix}
  2 \\
  1
\end{bmatrix}\\\nonumber
 &\left[\def\arraystretch{1.3} \begin{array}{c}
         \Upsilon_3^{\T}
         \\ \hline
         \Upsilon_4^{\T}
    \end{array}\right] =  \left[\setlength\arraycolsep{2pt}  \begin{array}{cccccc}
    0 & 0 & 0 & 0 & 0 & 0 \\ \hline
    0 & 0 & 0 & 0 & 0 & 0 
    \end{array}\right],\left[\begin{array}{c|c} B_2 & B_3 \\ \hline
B_4 & C_3
\end{array}\right] = \left[\begin{array}{c|c} 0 & 1 \\ \hline
0 & 1
\end{array}\right],  \\\nonumber
 &A_4^\top=B_1=C_1=\begin{bmatrix}
  1 &  0
\end{bmatrix},\Upsilon_2=
\text{diag}\left(\Upsilon_{d},\Upsilon_{d}\right),\\ \nonumber
    &\left[\def\arraystretch{1.3} \begin{array}{c}
         C_2  \\ \hline
         \Upsilon_1^{\T} \\ 
    \end{array}\right] =  \left[\setlength\arraycolsep{2pt}  \begin{array}{cccccc}
    0 & 0 & 0 & 0 & 0 & 0  \\ \hline
    1 & 0 & 0 & 0 & 0 & 0 \\
    0 & 0 & 0 & 1 & 0 & 0 \\ 
    \end{array}\right],\Upsilon_{d} = \begin{bmatrix}
  -(\rho+2) & 0 & 0\\
  -\rho & 1 & 0 \\
  0 & -\rho & 1 
\end{bmatrix}.
\end{eqnarray}
By applying the optimization problem \eqref{optm}, with \(\beta=-1.3\), we obtain the gain-scheduled SOF \eqref{blablala}, with matrices \(K_1=-1.2669\) and \(K_2=-1.3145\) that guarantees closed-loop stability with \(\mathcal{L}_2\)-gain bounded by \(\gamma=1.3493\).
\begin{figure}[t!]
    \centering
    \includegraphics[width=8cm]{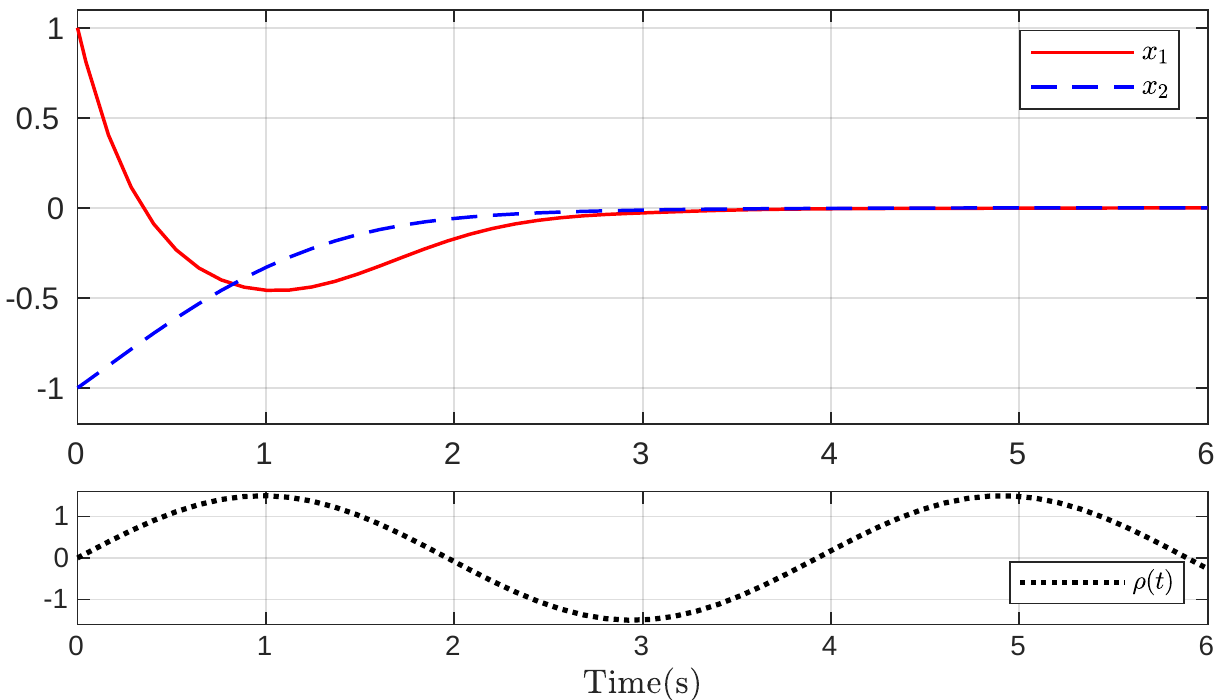}
    \caption{Example 1: States trajectories of the closed-loop system.}
    \label{fig:ex1}
\end{figure}
For simulations, we consider \(\rho(t)=1.5sin(1.6t)\), which can be rewritten as the convex combination of its vertices, such as \( \rho=-1.5\alpha_1 + 1.5 \alpha_2,\) with \(\alpha_1+\alpha_2=1\), leading to \(\alpha_2=1-\alpha_1\), \(\alpha_1=0.5-0.5sin(1.6t)\). Figure \ref{fig:ex1} presents simulation results of the closed-loop system for initial conditions \(x(0)=[1~~-1]^\top\).

\subsection{Example 2}
Consider an LPV system as follows,
\begin{equation}\label{ex2:eq1}
    \begin{cases}
&\dot{x}=
\begin{bmatrix}
  1+ \rho &  2-3\rho\\
  0 & -4-\rho
\end{bmatrix}x+
\begin{bmatrix}
  1 \\  \rho
\end{bmatrix}u + \begin{bmatrix}
  2- \rho \\  1
\end{bmatrix}w ,\\
&z=\begin{bmatrix} 
  1 & 2-\rho
\end{bmatrix}x + (1+\rho)u,\\
&y= 
\begin{bmatrix} 
  1+\rho & \rho
  \end{bmatrix}x.    
    \end{cases}
\end{equation}
with \(\rho \in [0~~~ 1]\). Replacing the limits of \(\rho\) in the system, we obtain the same two vertices system of Example 2 from \cite{behrouz2021robust}. Consider a DAR of system \eqref{ex2:eq1} with
\begin{eqnarray}\nonumber
&\pi=
\begin{bmatrix}
\rho x_1 & \rho x_2 & \rho u & \rho w
\end{bmatrix}^{\top},~B_3=1,~B_4=C_3=0, \\\nonumber
&A_1=
\begin{bmatrix}
  1 & 2\\
  0 & -4 \\
\end{bmatrix}, ~A_2=
\begin{bmatrix}
1 & -3 & 0 & -1\\
 0 & -1 & 1 & 0 
\end{bmatrix},~A_3=
\begin{bmatrix}
  1 \\
  0
\end{bmatrix},\\ \nonumber
&A_4=
\begin{bmatrix}
  0 \\
  1
\end{bmatrix},~B_1=\begin{bmatrix}
  1 \\
  2
\end{bmatrix}^{\top},\Upsilon_1^{\T} =  \left[\setlength\arraycolsep{2.5pt}  \begin{array}{cccccc}
    \rho & 0 & 0 & 0 \\
    0 & \rho & 0 & 0 \\ 
    \end{array}\right],\\\nonumber
 &B_2=\begin{bmatrix}
    0 & -1 & 1 & 0 
\end{bmatrix},~C_1^\top=A_3,~C_2=\begin{bmatrix}
  1 &  1 & 0 & 0
\end{bmatrix}, \\\nonumber
&\Upsilon_2=-I_4,~\Upsilon_3^{\T}=\begin{bmatrix}
  0 & 0 & \rho & 0
\end{bmatrix},~\Upsilon_4^{\T}=\begin{bmatrix}
  0 & 0 & 0 & \rho
\end{bmatrix}.
\end{eqnarray}
By applying the optimization problem \eqref{optm} with \(\beta=-29.3\), we obtain the gain-scheduled SOF \eqref{blablala}, with matrices \(K_1=-29.0522 ,~K_2=-29.2994\), that guarantees closed-loop stability with \(\mathcal{L}_2\)-gain bounded by \(\gamma= 5.2637\). Implementation of the control law is straightforward by following the same procedure illustrated in Example 1.

\section{Conclusion}

This paper proposed a new strategy based on strict QSR-dissipativity for gain-scheduled SOF stabilization of LPV systems with \(\mathcal{L}_2\)-gain performance. Finsler's Lemma and linear annihilators have been applied to formulate polytopic LMI conditions for dissipativity analysis and gain-scheduled SOF design. We successfully applied the strategy in two numeral examples. The first presents rational system matrices and the second presents affine system matrices on the time-varying parameter, both being open-loop unstable. The main contribution of this paper consists that the system matrices can present polynomial or rational dependence, not only affine as it is common in the field, and no restriction on the output plant matrix is considered. In addition, differently from some strategies in the field, the formulated solution does not need to solve a static feedback problem as initial stage to design the gain-scheduled static output feedback. 

\bibliographystyle{IEEEtran}
\bibliography{refs}

\begin{thebibliography}{10}
\providecommand{\url}[1]{#1}
\csname url@samestyle\endcsname
\providecommand{\newblock}{\relax}
\providecommand{\bibinfo}[2]{#2}
\providecommand{\BIBentrySTDinterwordspacing}{\spaceskip=0pt\relax}
\providecommand{\BIBentryALTinterwordstretchfactor}{4}
\providecommand{\BIBentryALTinterwordspacing}{\spaceskip=\fontdimen2\font plus
\BIBentryALTinterwordstretchfactor\fontdimen3\font minus
  \fontdimen4\font\relax}
\providecommand{\BIBforeignlanguage}[2]{{%
\expandafter\ifx\csname l@#1\endcsname\relax
\typeout{** WARNING: IEEEtran.bst: No hyphenation pattern has been}%
\typeout{** loaded for the language `#1'. Using the pattern for}%
\typeout{** the default language instead.}%
\else
\language=\csname l@#1\endcsname
\fi
#2}}
\providecommand{\BIBdecl}{\relax}
\BIBdecl

\bibitem{sadabadi2016static}
M.~S. Sadabadi and D.~Peaucelle, ``From static output feedback to structured
  robust static output feedback: A survey,'' \emph{Annual reviews in control},
  vol.~42, pp. 11--26, 2016.

\bibitem{trof2}
C.~A. Crusius and A.~Trofino, ``Sufficient {LMI} conditions for output feedback
  control problems,'' \emph{IEEE Transactions on Automatic Control}, vol.~44,
  no.~5, pp. 1053--1057, 1999.

\bibitem{vese1}
V.~Vesel\'y, ``Static output feedback controller design,'' \emph{Kybernetika},
  vol.~2, no.~37, pp. 205--221, 2001.

\bibitem{gahinet2011structured}
P.~Gahinet and P.~Apkarian, ``Structured $\text{H}_\infty$ synthesis in
  matlab,'' \emph{IFAC Proceedings Volumes}, vol.~44, no.~1, pp. 1435--1440,
  2011.

\bibitem{wei2014survey}
G.~Wei, Z.~Wang, W.~Li, and L.~Ma, ``A survey on gain-scheduled control and
  filtering for parameter-varying systems,'' \emph{Discrete Dynamics in Nature
  and Society}, vol. 2014, 2014.

\bibitem{apkarian1995self}
P.~Apkarian, P.~Gahinet, and G.~Becker, ``Self-scheduled $\text{H}_\infty$
  control of linear parameter-varying systems: a design example,''
  \emph{Automatica}, vol.~31, no.~9, pp. 1251--1261, 1995.

\bibitem{shamma1991guaranteed}
J.~S. Shamma and M.~Athans, ``Guaranteed properties of gain scheduled control
  for linear parameter-varying plants,'' \emph{Automatica}, vol.~27, no.~3, pp.
  559--564, 1991.

\bibitem{rugh2000research}
W.~J. Rugh and J.~S. Shamma, ``Research on gain scheduling,''
  \emph{Automatica}, vol.~36, no.~10, pp. 1401--1425, 2000.

\bibitem{peres2005}
V.~F. Montagner and P.~L.~D. Peres, ``{State Feedback Gain Scheduling for
  Linear Systems With Time-Varying Parameters},'' \emph{Journal of Dynamic
  Systems, Measurement, and Control}, vol. 128, no.~2, pp. 365--370, 2005.

\bibitem{al2018static}
A.~K. Al-Jiboory and G.~Zhu, ``Static output-feedback robust gain-scheduling
  control with guaranteed $\text{H}_2$ performance,'' \emph{Journal of the
  Franklin Institute}, vol. 355, no.~5, pp. 2221--2242, 2018.

\bibitem{nguyen2018gain}
A.-T. Nguyen, P.~Chevrel, and F.~Claveau, ``Gain-scheduled static output
  feedback control for saturated {LPV} systems with bounded parameter
  variations,'' \emph{Automatica}, vol.~89, pp. 420--424, 2018.

\bibitem{sereni2019new}
B.~Sereni, E.~Assun{\c{c}}{\~a}o, and M.~C.~M. Teixeira, ``New gain-scheduled
  static output feedback controller design strategy for stability and transient
  performance of {LPV} systems,'' \emph{IET Control Theory \& Applications},
  vol.~14, no.~5, pp. 717--725, 2019.

\bibitem{behrouz2021robust}
H.~Behrouz, I.~Mohammadzaman, and A.~Mohammadi, ``Robust static output feedback
  $\mathcal{H}_2/\mathcal{H}_\infty$ control synthesis with pole placement
  constraints: An {LMI} approach,'' \emph{International Journal of Control,
  Automation and Systems}, vol.~19, no.~1, pp. 241--254, 2021.

\bibitem{bouali2007new}
A.~Bouali, M.~Yagoubi, and P.~Chevrel, ``New {LMI-based} conditions for
  stability, $\mathcal{H}_\infty$ performance analysis and state-feedback
  control of rational {LPV} systems,'' in \emph{2007 European Control
  Conference (ECC)}.\hskip 1em plus 0.5em minus 0.4em\relax IEEE, 2007, pp.
  5411--5417.

\bibitem{bouali2008}
{A. Bouali, M. Yagoubi and P. Chevrel}, ``$\mathcal{H}_2$ gain scheduling
  control for rational {LPV} systems using the descriptor framework,'' in
  \emph{2008 47th IEEE Conference on Decision and Control}.\hskip 1em plus
  0.5em minus 0.4em\relax IEEE, 2008, pp. 3878--3883.

\bibitem{masubuchi2008gain}
I.~Masubuchi and A.~Suzuki, ``Gain-scheduled controller synthesis based on new
  {LMIs} for dissipativity of descriptor {LPV} systems,'' \emph{IFAC
  Proceedings Volumes}, vol.~41, no.~2, pp. 9993--9998, 2008.

\bibitem{polcz2020induced}
P.~Polcz, T.~P{\'e}ni, B.~Kulcsar, and G.~Szederk{\'e}nyi, ``Induced {L2}-gain
  computation for rational {LPV} systems using {Finsler’s} lemma and minimal
  generators,'' \emph{Systems \& Control Letters}, vol. 142, p. 104738, 2020.

\bibitem{sadeghzadeh2017gain}
A.~Sadeghzadeh, ``Gain-scheduled static output feedback controller synthesis
  for discrete-time {LPV} systems,'' \emph{International Journal of Systems
  Science}, vol.~48, no.~14, pp. 2936--2947, 2017.

\bibitem{de2010gain}
J.~De~Caigny, J.~F. Camino, R.~C. Oliveira, P.~L.~D. Peres, and J.~Swevers,
  ``Gain-scheduled $\text{H}_2$ and $\text{H}_\infty$ control of discrete-time
  polytopic time-varying systems,'' \emph{IET control theory \& applications},
  vol.~4, no.~3, pp. 362--380, 2010.

\bibitem{peixoto2020}
M.~L. Peixoto, P.~S. Pessim, M.~J. Lacerda, and R.~M. Palhares, ``Stability and
  stabilization for {LPV} systems based on lyapunov functions with
  non-monotonic terms,'' \emph{Journal of the Franklin Institute}, vol. 357,
  no.~11, pp. 6595--6614, 2020.

\bibitem{peixoto2021improved}
M.~L. Peixoto, P.~H. Coutinho, and R.~M. Palhares, ``Improved robust
  gain-scheduling static output-feedback control for discrete-time {LPV}
  systems,'' \emph{European Journal of Control}, vol.~58, pp. 11--16, 2021.

\bibitem{willm2}
J.~Willems, ``A system theory approach to unified electrical machine
  analysis,'' \emph{International Journal of Control}, vol.~15, no.~3, pp.
  401--418, 1972.

\bibitem{brogl1}
B.~Brogliato, R.~Lozano, B.~Maschke, and O.~Egeland, \emph{Dissipative Systems
  Analysis and Control - Theory and Applications}.\hskip 1em plus 0.5em minus
  0.4em\relax London, UK: Springer-Verlag, 2020.

\bibitem{madeira2021}
{D. de S. Madeira}, ``Necessary and sufficient dissipativity-based conditions
  for feedback stabilization,'' \emph{IEEE Transactions on Automatic Control,
  {doi:} 10.1109/TAC.2021.3074850}, 2021.

\bibitem{madeira2020application}
{D. de S. Madeira} and V.~V. Viana, ``An application of {QSR}-dissipativity to
  the problem of static output feedback robust stabilization of nonlinear
  systems,'' \emph{Anais da Sociedade Brasileira de Autom{\'a}tica}, vol.~2,
  no.~1, 2020.

\bibitem{briat2014}
C.~Briat, \emph{Linear Parameter-Varying and Time-Delay Systems: Analysis,
  Observation, Filtering \& Control}.\hskip 1em plus 0.5em minus 0.4em\relax
  Springer, 2014.

\bibitem{coutinho2009robust}
D.~F. Coutinho, C.~E. de~Souza, and K.~A. Barbosa, ``Robust {$H_\infty$} filter
  design for a class of discrete-time parameter varying systems,''
  \emph{Automatica}, vol.~45, no.~12, pp. 2946--2954, 2009.

\bibitem{trofi2}
A.~Trofino and T.~Dezuo, ``{LMI} stability conditions for uncertain rational
  nonlinear systems,'' \emph{International Journal of Robust and Nonlinear
  Control}, vol.~24, no.~18, pp. 3124--3169, 2014.

\bibitem{coutinho2008}
D.~Coutinho, M.~Fu, A.~Trofino, and P.~Danes, ``$\mathcal{L}_2$-gain analysis
  and control of uncertain nonlinear systems with bounded disturbance inputs,''
  \emph{International Journal of Robust and Nonlinear Control: IFAC-Affiliated
  Journal}, vol.~18, no.~1, pp. 88--110, 2008.

\bibitem{haddad2011nonlinear}
W.~M. Haddad and V.~Chellaboina, \emph{Nonlinear dynamical systems and control:
  a Lyapunov-based approach}.\hskip 1em plus 0.5em minus 0.4em\relax Princeton,
  NJ, USA: Princeton university press, 2008.

\bibitem{cao1998static}
Y.-Y. Cao, J.~Lam, and Y.-X. Sun, ``Static output feedback stabilization: an
  {ILMI} approach,'' \emph{Automatica}, vol.~34, no.~12, pp. 1641--1645, 1998.

\bibitem{lofberg2004yalmip}
J.~Lofberg, ``{YALMIP} : a toolbox for modeling and optimization in {MATLAB},''
  in \emph{IEEE International Conference on Robotics and Automation}, New
  Orleans, LA, USA, 2004, pp. 284--289.

\bibitem{sturm1999using}
J.~F. Sturm, ``Using {Sedumi} 1.02, a {Matlab} toolbox for optimization over
  symmetric cones,'' \emph{Optimization methods and software}, vol.~11, no.
  1-4, pp. 625--653, 1999.

\end{thebibliography}

\end{document}